\documentclass[12pt]{amsart}

\usepackage[english]{babel}
\usepackage[latin1]{inputenc}
\usepackage{amsmath,latexsym,amsmath}
\usepackage{amsfonts}
\usepackage{amssymb}
\usepackage{geometry}
\usepackage{graphicx}
\usepackage{amsthm}
\usepackage{bm}
\usepackage{setspace}
\usepackage[final]{showkeys}
\usepackage{dsfont}
\usepackage{enumitem}

\usepackage{epsfig}
\usepackage{graphicx}
\usepackage{psfrag}
\usepackage{enumitem}
\usepackage{bbm}
\usepackage{hyperref}

\theoremstyle{definition}
\newtheorem{The}{Theorem}
\newtheorem{Def}[The]{Definition}
\newtheorem{Cor}[The]{Corollary}
\newtheorem{Pro}[The]{Proposition}
\newtheorem{Lem}[The]{Lemma}

\newcommand{\E}{\mathbb{E}}
\newcommand{\lip}{{\rm Lip}}

\begin{document}

\title[Concentration inequalities for sequential dynamical systems]{Concentration inequalities for sequential dynamical systems of the unit interval}

\author{Romain Aimino}\address{
Dipartimento di Matematica, II Universit\`a di Roma (Tor Vergata), Via della Ricerca Scientifica, 00133 Roma, Italy.
}
\email{\href{mailto:aimino@mat.uniroma2.it}{aimino@mat.uniroma2.it}}
\urladdr{\url{http://www.mat.uniroma2.it/~aimino/}}

\author{J\'er\^{o}me Rousseau}\address{J\'er\^ome Rousseau, Departamento de Matem\'atica, Universidade Federal da Bahia\\
Av. Ademar de Barros s/n, 40170-110 Salvador, BA, Brazil}
\email{\href{mailto:jerome.rousseau@ufba.br}{jerome.rousseau@ufba.br}} 
\urladdr{\url{http://www.sd.mat.ufba.br/~jerome.rousseau}}

\date{\today}

\begin{abstract} We prove a concentration inequality for sequential dynamical systems of the unit interval enjoying an exponential loss of memory in the BV norm, and we investigate several of its consequences. In particular, this covers compositions of $\beta$-transformations, with all $\beta$ lying in a neighborhood of a fixed $\beta_{\star} > 1$ and systems satisfying a covering type assumption.
\end{abstract}

\maketitle

\tableofcontents

\section{Introduction}

Sequential dynamical systems consist of a composition of different maps from a space into itself, and model discrete-time non-autonomous dynamics, where the law of evolution is changing, possibly slowly, along the time. Such systems have been studied either from a topological point of view, with for instance the introduction of the notion of topological entropy \cite{KS96} and Devaney's chaos  \cite{ChSh09}, or from an ergodic theoretic point of view: \cite{BB84} has developed the notions of ergodicity and mixing, and \cite{K14} defined a measure theoretic entropy.

It is important to notice that in general, there is no common invariant measure for all maps, even though there is often a natural probability measure on the space, such as the Lebesgue measure, allowing to look at random processes that arise from observations. But this lack of invariance of the measure and the fact that maps change with the time imply that the processes are non-stationary, adding difficulties to their study. 

Up to now, the study of the statistical properties remains incomplete and the emphasis has been mainly put on the property of loss of memory, which is the generalization of the classical notion of decay of correlations to sequential systems: if we start from two probability densities, both belonging to a class of smooth enough functions, and if we let them evolve with the dynamic, they will be attracted one from each other if the system is sufficiently mixing. This leads to the idea that all smooth densities are attracted indifferently by the same moving target in the space of densities, whence the name of loss of memory. Many works have been devoted to the study of the speed of loss of memory. Among them, we can cite \cite{OSY09} for smooth expanding maps, \cite{OSY09} and \cite{Conze_Raugi} for one-dimensional piecewise expanding maps, \cite{GOT13} for multidimensional piecewise expanding maps, \cite{S11} for two-dimensional Anosov systems, \cite{SYZ13} for billiards with moving convex scatterers, \cite{
MO14} for one-dimensional open systems, and \cite{AHNTV} for maps of the interval with a neutral fixed point, the latter reference dealing with systems enjoying a polynomial loss of memory, whereas all the others examples exhibit an exponential speed of decorrelation.

Beyond the loss of memory, only the central limit theorem has been studied in \cite{Bakh95I,Bakh95II} for some classes of hyperbolic maps, and in \cite{Conze_Raugi} and \cite{NSV12} for one-dimensional piecewise expanding maps. It is worthy to note that \cite{Conze_Raugi} investigates also the strong law of large numbers and dynamical Borel-Cantelli lemmas. Very recently, the almost sure invariance principle has been studied in \cite{HNTV}

Despite this relative sparseness of the literature on statistical properties for non-autonomous systems, as written by Stenlund \cite{S11}: {\em "Much of the statistical theory of stationary dynamical systems can be carried over to sufficiently chaotic non-stationary systems."}

Following this philosophy, we address in this paper the question of proving concentration inequalities for sequential dynamical systems. The main object of concentration inequalities, a well known subject in probability theory, is to estimate the deviation from the mean of random variables of the form $Y_n = K(X_0, \ldots, X_{n-1})$ which depends in a smooth way of the variables $X_n$, even though the dependence could be complicated or implicit. Such observables arise naturally in statistical applications. Concentration inequalities are then an extension of large deviations inequalities for ergodic sums to more general observables. It should be noted that they are also non-asymptotic, in contrast to standard large deviation principles, however they do not give precise asymptotic, with sharp bounds. We refer the reader to \cite{BLM, Ledoux, McDi1, McDi2} among others for reviews of this field in a pure probabilistic setting.

Concentration inequalities were brought into a dynamical context in 2002 by Collet, Martinez and Schmitt \cite{CMS} who proved such an inequality for one-dimensional piecewise expanding maps of the interval. More precisely, they showed that if $T$ is a piecewise $C^2$ map of the unit interval which is uniformly expanding and topologically mixing, then the process $(T^n)_n$, defined on the probability space $([0,1], \mu)$, where $\mu$ is the unique absolutely continuous $T$-invariant measure, satisfies a concentration inequality, provided the density of $\mu$ is bounded away from $0$. Subsequently, a lot of works have followed, aiming to extend this result to non-uniformly expanding/hyperbolic systems \cite{CCSI, CCRV, CG}, or to random dynamical systems \cite{Mal12, ANV}.A rather complete description of the situation for systems modeled by a Young tower with an exponential or polynomial tail of the return time is given in \cite{CG} and \cite{GM}.

In this paper, we focus on the class of sequential dynamical systems described by Conze and Raugi \cite{Conze_Raugi}: they consist of compositions of piecewise expanding maps $(T_n)$ of the unit interval that enjoy an exponential loss of memory in the BV norm. Under this decorrelation assumption, together with a minoration hypothesis on the evolution of the Lebesgue measure $m$ under the dynamics, which was shown to hold in \cite{Conze_Raugi} for compositions of $\beta$-transformations with all $\beta$ very close, we proved that the process $( T_n \circ \ldots \circ T_1)_n$ defined on the probability space $([0,1], m)$ satisfies an exponential concentration inequality, and we discuss some of its consequences.
\\ ~

{\bf Outline of the paper.} In Section \ref{section:background}, we describe the class of sequential systems introduced by Conze and Raugi \cite{Conze_Raugi}, and discuss more in details our assumptions. In Section \ref{section:main}, we state and prove our main result: an exponential concentration inequalities for sequential dynamical systems. In Section \ref{section:app} we give several applications of concentration to the empirical measure, the shadowing and the almost sure central limit theorem. In Section \ref{secaltappr}, we prove a large deviation estimate for observables belonging to BV and an upper bound for the Kantorovich distance using a direct approach. The Appendix \ref{app:formula} is devoted to the proof of a formula used to establish the main result.

\section{Sequential dynamical systems of the unit interval} \label{section:background}

In this section, we recall some background on sequential dynamical systems. We follow closely the lines of Conze and Raugi \cite{Conze_Raugi}, and we refer to their paper for more details and results.

\subsection{Generalities on sequential dynamical systems}

Let $(X, \mathcal{F}, m)$ be a probability space. A sequential dynamical system on $(X,m)$ is a sequence $(T_n)_{n \ge 1}$ of non-singular transformations \footnote{Recall that a non-singular transformation on a probability space $(X, \mathcal{F}, m)$ is a measurable map $T : X \to X$ such that $T_{\star} m \ll m$.} on $(X, \mathcal{F}, m)$

For $n \le k$, we denote by $T_n^k$ the composition $T_n^k = T_k \circ \ldots \circ T_n$. For a non-singular map $T$, we denote by $P_T$ the transfer operator, it acts on $L^1(m)$, and satisfies for all $f \in L^1(m)$ and $g \in  L^{\infty}(m)$: 
$$\int P_T f . g \, dm = \int f .g \circ T \, dm.$$ 

We clearly have $P_{T_n^k} = P_{T_k} \circ \ldots \circ P_{T_n}$. For notational convenience, we will drop the letter $T$ when the sequence of transformations is understood, and denote by $P_n^k$ the operator $P_{T_n^k}$.

For a given sequential dynamical system $(T_n)$, we form the decreasing sequence of $\sigma$-algebras $$\mathcal{F}_n = (T_1^n)^{-1}(\mathcal{F}),$$ and define the asymptotic $\sigma$-algebra by $$\mathcal{F}_{\infty} = \bigcap_{n \ge 1} \mathcal{F}_n.$$

Let $f\in L^1(m)$. Since $P_1^n f(x) = 0$ a.s. on the set $\{P_1^n \mathds{1} = 0 \}$, we define the quotient $\frac{P_1^n f}{P_1^n \mathds{1}}$ as $0$ on $\{P_1^n \mathds{1} = 0 \}$. Then, we have the relation $$\mathbb{E}_m(f | \mathcal{F}_n) = \left( \frac{P_1^n f}{P_1^n \mathds{1}}\right) \circ T_1^n.$$

By Doob's convergence theorem for martingales, the sequence of conditional expectations $\mathbb{E}_m(f | \mathcal{A }_n)$ converges a.s. and in $L^1(m)$ to $\mathbb{E}_m(f|\mathcal{F}_{\infty})$. We say that the sequential dynamical system $(T_n)$ is exact when its associated asymptotic $\sigma$-algebra $\mathcal{F}_{\infty}$ is trivial modulo $m$. Equivalently, the system is exact if $\lim_n \| P_1^n f \|_{L^1_m} = 0$ for all $f \in L^1(m)$ with $\int f \, dm = 0$.

\subsection{A functional analytic framework}

In \cite{Conze_Raugi}, Conze and Raugi have extended the spectral theory of the iterates of a single operator to the case of concatenations of different operators, having in mind applications to sequential dynamical systems. We recall briefly the setting and the main results.

Let $(\mathcal{B}, \| . \|)$ be a Banach space, $\mathcal{V}$ be a subspace of $\mathcal{B}$ equipped with a norm $| . |_v$ such that $\| . \| \le | . |_v$. Let $\mathcal{P}$ be a set of contractions  of $(\mathcal{B}, \|.\|)$ \footnote{i.e. a set of linear operators $P : \mathcal{B} \to \mathcal{B}$ satisfying $\|P f \|  \le \|f \|$ for all $f \in \mathcal{B}$.} leaving $\mathcal{V}$ invariant, and satisfying 

\begin{enumerate}

\item[(H1)] The unit ball of $(\mathcal{V}, |.|_v)$ is relatively compact in $(\mathcal{B}, \|.\|)$.  

\item[(H2)] There is a countable family in $\mathcal{V}$ which is dense in $(\mathcal{B}, \|.\|)$.  

\item[(H3)] There are an integer $r \ge 1$ and constants $0 < \rho_r < 1$, $M_0, C_r > 0$ such that:$$ \forall P \in \mathcal{P}, \: |Pf|_v \le M_0 |f|_v, \: \forall f \in \mathcal{V};$$ and for all $r$-tuples $P_1, \ldots, P_r$ of operators in $\mathcal{P}$: $$\forall f \in \mathcal{V}, \: |P_r \ldots P_1 f |_v \le \rho_r |f|_v + C_r \|f\|.$$

\end{enumerate}

They consider mainly the case of one-dimensional systems, where $(X,\mathcal{F},m)$ is the unit interval $[0,1]$ endowed with its Borel $\sigma$-algebra and the Lebesgue measure, and where the maps $T$ are piecewise $C^2$ and uniformly expanding. In this setting, the class $\mathcal{P}$ will consist of the transfer operators of the class of maps considered, and the natural choice for the Banach spaces is to take $\mathcal{B} = L^1(m)$ and $\mathcal{V} = {\rm BV}$ the space of functions of bounded variations.
The previous assumptions imply that there exists $M > 0$ such that $|P_n \ldots P_1 f |_v \le M |f|_v$ for all $n \ge 1$, all $f \in \mathcal{V}$ and all choices of operators $P_1, \ldots, P_n$ in $\mathcal{P}$.

Let $\mathcal{V}_0$ be a subspace of $\mathcal{V}$ which is invariant by all operators in $\mathcal{P}$. It will consist in concrete applications of the functions in $\mathcal{V}$ with zero average with respect to $m$. We say that a sequence of operators $(P_n)$ in $\mathcal{P}$ is exact if for all $f \in \mathcal{V}_0$, $$\lim_{n} \|P_n \ldots P_1 f \| = 0.$$ A single operator $P \in \mathcal{P}$ is exact in $\mathcal{V}_0$ if the sequence given by $P_n = P$ is exact in $\mathcal{V}_0$. When $P_n$ is the transfer operator of a transformation $T_n$ and $\mathcal{V}_0$ is the set of functions in $\mathcal{V}$ with zero average, then exactness of the sequence $(P_n)$ in $\mathcal{V}_0$ implies exactness of the sequential dynamical system $(T_n)$ as defined previously, as easily deduced from the facts that $\mathcal{V}$ is dense in $(\mathcal{B}, \|.\|)$ and that all operators in $\mathcal{P}$ are contractions on $\mathcal{B}$.

\subsection{Exponential loss of memory}

It is well known since the work of Ionescu-Tulcea and Marinescu \cite{ITM} that each single operator in $\mathcal{P}$ is quasi-compact on $\mathcal{V}$ and enjoys a nice spectral decomposition, see Proposition 2.8 in \cite{Conze_Raugi} for a self-contained proof, and \cite{Bal00, BG_book, HH} among others for further properties. As a direct consequence, it follows that if an operator $P \in \mathcal{P}$ is exact in $\mathcal{V}_0$, then the norm of $P^n$ seen as an operator on $\mathcal{V}_0$ decays exponentially fast: there exist $K > 0$ and $\theta < 1$ such that $|P^n f|_v \le K \theta^n |f|_v$ for all $n \ge 1$ and all $f \in \mathcal{V}_0$.

For compositions of different operators, the situation is slightly more difficult, and we need more assumptions in order to get exponential decay. First we state the decorrelation property: 
\begin{description}
\item[{\em  (Dec)}] a subset $\mathcal{P}_0 \subset \mathcal{P}$ satisfies the decorrelation property in $\mathcal{V}_0$ if there exist $\theta < 1$ and $K > 0$ such that, for all integers $l \ge 1$, all $l$-tuples of operators $P_1, \ldots, P_l$ in $\mathcal{P}_0$: $$\forall f \in \mathcal{V}_0, \: |P_l \ldots P_1 f |_v \le K \theta^l |f|_v.$$
\end{description}
Conze and Raugi \cite{Conze_Raugi} have given a condition ensuring that {\em (Dec)} is verified. Rather than stating their condition, we give two corollaries from their paper.

The first one, which is of a local nature, states that any exact operator $P$ admits a convenient neighborhood for which {\em (Dec)} holds. More precisely, for two operators $P, P' \in \mathcal{P}$, define 
$$d(P,P') = \sup_{\{f \in \mathcal{V} \, : \, |f|_v \le 1\}} \|Pf - P'f\|,$$
and for $\delta  > 0$, denote $B(P, \delta) = \{P' \in \mathcal{P} \, : \, d(P, P') < \delta \}$.

Then, Proposition 2.10 from \cite{Conze_Raugi} asserts that for all $P \in \mathcal{P}$ exact in $\mathcal{V}_0$, there exists a $\delta_0 > 0$ such that $\mathcal{P}_0 = \mathcal{P} \cap B(P, \delta_0)$ satisfies {\em (Dec)} in $\mathcal{V}_0$.

The other corollary, which does not need any sort of closeness, requires a compactness condition: a subset $\mathcal{P}_0 \subset \mathcal{P}$ satisfies the compactness condition {\bf (C)} if for any sequence $(P_n)$ in $\mathcal{P}_0$, there exist a subsequence $(P_{n_j})$ and an operator $P \in \mathcal{P}_0$ such that $$\forall f \in \mathcal{B}, \: \lim_j \| P_{n_j} f - P f \| = 0.$$
Proposition 2.11 in \cite{Conze_Raugi} says that if $\mathcal{P}_0$ satisfies the compactness condition {\bf (C)} and is such that all sequences in $\mathcal{P}_0$ are exact in $\mathcal{V}_0$, then it satisfies the decorrelation property {\em (Dec)} in $\mathcal{V}_0$.

\subsection{Application to one-dimensional systems} 

We now describe how this theory applies to concrete situations, with piecewise expanding maps on $X = [0,1]$. More precisely, we consider maps $T : X \to X$ uniformly expanding, i.e. $ \lambda(T) := \inf_x |T'(x)| > 1$, and such that there exists a finite partition $\mathcal{A}_T$ of $X$ consisting of intervals with disjoint interiors, such that the map $T$ can be extended to a $C^2$ map on a neighborhood of each element of the partition. It is well know that the transfer operator $P_T : L^1(m) \to L^1(m)$ of $T$ satisfies $$P_T  f(x) = \sum_{Ty = x} \frac{f(y)}{|T'(y)|},$$  for all $f \in L^1(m)$, and is quasi-compact on the space ${\rm BV}$ of functions of bounded variation on $X$, see Baladi \cite{Bal00} or Boyarsky and G\'ora \cite{BG_book}. Recall that for $f \in L^1(m)$, we define $${\rm V}f = \inf \{\: {\rm var} \bar{f} \, : \, f = \bar{f} \, \, m{\rm -a.e.}\},$$ where $$ {\rm var} \bar{f} = \sup \left\{ \sum_{j=0}^{t-1} |\bar{f} (x_{j+1}) - \bar{f} (x_j) | \, : \, 0 = x_0 < \cdots < x_t = 1 \right\}.
$$
The space ${\rm BV}$ is equipped with the norm $\| . \|_{\rm BV} = {\rm V}(.) + \| . \|_{L^1_m}$ and is a Banach space whose unit ball is compact in $L^1(m)$.

The couple $(\mathcal{B}, \| . \|) = (L^1(m), \|.\|_{L^1_m})$ and $(\mathcal{V}, |.|_v) = ({\rm BV}, \|.\|_{\rm BV})$ satisfies (H1) and (H2). To ensure (H3), we first need to recall the so-called Lasota-Yorke inequality \cite{LY}: for any piecewise expanding map $T$ and any $f \in {\rm BV}$, we have $${\rm V}(P_T f) \le \frac{2}{\lambda(T)} {\rm V} f + C(T) \|f\|_{L^1_m},$$ with $C(T) = \sup_x \frac{|T''|}{|T'|^2} + 2 \sup_{I \in \mathcal{A}_T} \frac{\sup_{x \in I} |T'(x)|^{-1}}{m(I)}$.

Then, any class $\mathcal{C}$ of piecewise expanding maps $T$ for which $\sup_{T \in \mathcal{C}} C(T) < \infty$ and for which there exists a $r \ge 1$ such that $$\inf_{T_1, \ldots, T_r \in \mathcal{C}} \lambda(T_1^r) > 2, \: {\rm and } \: \sup_{T_1, \ldots, T_r \in \mathcal{C}} C(T_1^r) < \infty,$$ will satisfy (H3) with $\mathcal{P}$ the class of transfer operators of all maps in $\mathcal{C}$. From now on, we will say that
\begin{description}
\item[$\boldsymbol{(D_r)}$] a class $\mathcal{C}$ satisfies this condition if the corresponding set of transfer operators satisfies (H3) with this particular $r$. 
\end{description}

The class of maps described above hence satisfies $(D_r)$.

It is shown in \cite{Conze_Raugi} that the class of $\beta$-transformations, i.e. maps of the form $x\mapsto \beta x$ mod $1$, with $\beta \ge 1 + \alpha$, $\alpha > 0$, satisfies $(D_r)$ for some $r$ depending only on $\alpha$, see Theorem 3.4 c) therein. They also proved that this class of transformations is exact (Theorem 3.6), and that $d(P_1, P_2) \le C |\beta_1 - \beta_2|$ for some universal constant $C$, where $\beta_i > 1$ is arbitrary and $P_i$ is the transfer operator of the $\beta_i$-transformation, for $i=1,2$ (Lemma 3.9). 

As a direct consequence, the class of $\beta$-transformations, with $\beta$ in between $\beta_{\rm min} > 1$ and $\beta_{\rm max} < \infty$ satisfies the compactness condition {\bf (C)} and hence verifies the decorrelation property {\em (Dec)}.

\subsection{Limit theorems}

We state the two main limit theorems from \cite{Conze_Raugi}. Let $\mathcal{C}$ be a class of transformations that satisfies $(D_r)$ and {\em (Dec)}. Then, for any sequence of maps $(T_n)_{n \ge 1}$ in $\mathcal{C}$, we have the following strong law of large numbers (Theorem 3.7): for any $f \in {\rm BV}$ 
$$\lim_{n\rightarrow\infty} \frac{S_n}{n} =0, \: m {\rm -a.s.}$$ 
where $S_n = \sum_{k=0}^{n-1} \left[ f \circ T_1^k - \int f \circ T_1^k \, dm \right]$.

In order to have the corresponding central limit theorem, one more assumption is needed:
\begin{description}
\item[(Min)] the sequence of maps $(T_n)$ satisfies this condition if there exists a $\delta > 0$ such that $P_1^n \mathds{1} \ge \delta$ for all $n$. 
\end{description}
Note that this condition concerns only a particular sequence, and not the whole class $\mathcal{C}$.

Moreover, if $(T_n)$ be a sequence of maps which verifies  {\bf (Min)} and belongs to $\mathcal{C}$. Then (Theorem 5.1), for any $f \in {\rm BV}$, either $\|S_n\|_2$ is bounded and in this case $S_n$ is bounded almost surely, or $\|S_n\|_2$ is unbounded, and in this case, $\|S_n\|_2$ goes to $+\infty$ and we have 
$$\frac{S_n}{\|S_n\|_2}\overset{\mathcal{L}}{\rightarrow} \mathcal{N}(0,1).$$

\subsection{The minoration condition} 

In this section, we discuss the condition {\bf (Min)}. In the case of $\beta$ transformations, Conze and Raugi have shown that for all $\beta > 1$, there exists a neighborhood $[\beta-a, \beta+a]$ of $\beta$ such that if all maps $T_n$ are $\beta_n$-transformations, with $\beta_n \in [\beta -a, \beta +a]$, then {\bf (Min)} holds true.

We now give a condition that automatically ensures the validity of {\bf (Min)}. If $(T_n)$ is a sequential dynamical system on $[0,1]$ made of piecewise monotonic maps, we denote by $\mathcal{A}_n$ the partition of monotonicity of $T_n$, and by $\mathcal{A}_n^m$ the partition of $T_n^m$, $n \le m$.

\begin{Def} \em The sequential dynamical system $(T_n)$ is covering if 
\[ \forall n, \, \exists N(n) \: : \: \forall m, \, \forall I \in \mathcal{A}_{m+1}^{m+n}, \: \:  T_{m+1}^{m+N(n)}(I)= [0,1].\]
\end{Def}

Let $\mathcal{C}$ be a class of surjective piecewise expanding maps of $[0,1]$ that satisfies $(D_r)$ for some $r \ge 1$ and for which there exist $\lambda > 1$ and $C > 0$ such that $\inf |T'| \ge \lambda$ and $\sup |T'| \le C$ for all $T \in \mathcal{C}$.

\begin{Pro} \em Let $(T_n)$ be a sequential dynamical system belonging to a class $\mathcal{C}$ as before, and suppose that $(T_n)$ is covering. Then the condition {\bf (Min)} is satisfied.
\end{Pro}

\begin{proof}
For $a > 0$, define \[\mathcal{E}_a =  \left\{f \in {\rm BV} \, : \, f \ge 0, \,  f \neq 0, \,\|f \|_{\rm BV} \le a \int f \right\}. \]
Recall from Lemma 3.2 in Liverani \cite{Liv95}, that given a partition $\mathcal{A}$ of $[0,1]$ made of intervals of length less than $\frac{1}{2a}$, for any $f \in \mathcal{E}_a$, there exists $I \in \mathcal{A}$ such that $f(x) \ge \frac{1}{2} \int f$ for all $x \in I$.
Since $\mathcal{C}$ satisfies $(D_r)$, for any $P_1, \ldots, P_r$ transfer operators associated to maps in $\mathcal{C}$ and any $f \in \mathcal{E}_a$, we have $$\|P_r \ldots P_1 f \|_{\rm BV} \le \rho_r \|f\|_{\rm BV} + C_r \|f \|_{L^1_m} \le \left( a \rho_r + C_r\right) \int f.$$
Hence, if we choose $a \ge \frac{C_r}{1 - \rho_r}$, we have $(P_r \ldots P_1) (\mathcal{E}_a) \subset \mathcal{E}_a$ for any choice of $P_1, \ldots, P_r$. In order to also have $\mathds{1 } \in \mathcal{E}_a$, we will actually choose $a = \max \{1, \frac{C_r}{1 - \rho_r} \}$.

Let now $(T_n)$ be a sequence of maps in $\mathcal{C}$ which is covering. For any $n \le m$, the diameter of $\mathcal{A}_n^m$ is less than $\lambda^{-(m-n+1)}$, so we can choose $n_0 \ge 1$ such that the diameter of $\mathcal{A}_{n+1}^{n+n_0}$ is less than $\frac{1}{2a}$ for all $n$.

Let $m \ge 0$. We have $P_1^{m+N(n_0)} \mathds{1} = P_{m+1}^{m+ N(n_0)} P_1^m \mathds{1}$. Write $m = p_m r + q_m$, with $0 \le q_m < r$. We have 
\begin{equation}\label{eqp1q}
 P_1^{q_m} \mathds{1}(x) = \sum_{T_1^{q_m}y=x} \frac{1}{|(T_1^{q_m})'(y)|} \ge C^{-q_m} \ge C^{-r},
\end{equation}
since all maps are surjective and have a derivative uniformly bounded by $C$. As a consequence, we have $P_1^{m+N(n_0)} \mathds{1} \ge C^{-r} P_{m+1}^{m+N(n_0)} g_m$, with $g_m = P_{q_m+1}^{m} \mathds{1}$. Since $P_{q_m+1}^{m} \mathds{1} $ is a concatenation of $p_m$ blocks of $r$ operators applied to a function in $\mathcal{E}_a$, we obtain that $g_m$ belongs to $\mathcal{E}_a$. There whence exists an interval $I_m \in \mathcal{A}_{m+1}^{m+n_0}$ on which $g_m \ge \frac{1}{2}$. This implies 
$$P_1^{m+N(n_0)} \mathds{1}(x) \ge \frac{C^{-r}}{2} P_{m+1}^{m+N(n_0)} \mathds{1}_{I_m}(x) = \frac{C^{-r}}{2} \sum_{T_{m+1}^{m+N(n_0)} y = x} \frac{\mathds{1}_{I_m}(y)}{|(T_{m+1}^{m+N(n_0)})'(y)|} \ge \frac{C^{-r}}{2}C^{-N(n_0)},$$
since for all $x \in [0,1]$, there exists a $y \in I_
m$ such that $T_{m+1}^{m+N(n_0)}y =x$, by the covering assumption. Using the same argument as in \eqref{eqp1q}, $P_1^{k} \mathds{1} \ge C^{-N(n_0)}$ for all $k < N(n_0)$ and this shows that {\bf (Min)} is satisfied with $\delta = \min \{C^{-N(n_0)}, \frac{C^{-(r+N(n_0))}}{2} \}$.
\end{proof}

\section{A concentration inequality} \label{section:main}

Let $(X,d)$ be a metric space. A function $K : X^n \to \mathbb{R}$ is said to be separately Lipschitz if for all $i = 0, \ldots, n-1$, there exists a constant $\lip_i(K)$ such that 
$$ \left| K(x_0, \ldots, x_{i-1}, x_i, x_{i+1}, \ldots, x_{n-1}) - K(x_0, \ldots, x_{i-1}, x_i', x_{i+1}, \ldots, x_{n-1}) \right| \le \lip_i(K) d(x_i, x_i')$$
for all points $x_0, \ldots, x_{n-1}, x_i'$ in $X$.

Let $(Z_n)_{n \ge 0}$ be a discrete time random process with values in $X$. We say that this process satisfies an exponential concentration inequality if there exists $C > 0$ such that for all $n \ge 1$ and all functions $K : X^n \to \mathbb{R}$ separately Lipschitz, one has $$\mathbb{E}(e^{K(Z_0, \ldots, Z_{n-1}) - \mathbb{E}(K(Z_0, \ldots, Z_{n-1}))}) \le e^{C \sum_{j=0}^{n-1} \lip_j^2(K)}.$$

This implies a large deviation estimate: for all $t > 0$, 

\[
\mathbb{P}\left( K(Z_0, \ldots, Z_{n-1}) - \mathbb{E}(K(Z_0, \ldots, Z_{n-1}) > t \right) \le e^{- \frac{t^2}{4C \sum_{j=0}^{n-1} \lip_j^2(K)}},
\]
which follows by optimizing over $\lambda > 0$ the inequality $\mathbb{P}( Y > t) \le e^{-\lambda t} \mathbb{E}(e^{\lambda Y})$, with $Y = K(Z_0, \ldots, Z_{n-1}) - \mathbb{E}(K(Z_0, \ldots, Z_{n-1})$ (e.g. \cite{CG}).

We will consider processes generated by sequential dynamical systems on the unit interval, that is process of the form $Z_n = T_1^n$, where $(T_n)_{n \ge 1}$ is a sequential dynamical system on $[0,1]$. These processes are defined on the probability space $([0,1], \mathcal{F}, m$), where $m$ is the Lebesgue measure defined on the Borel $ \sigma$-algebra, which is endowed with the usual distance $d(x,y) = |x - y|$.

The maps $T_n$ will belong to a class $\mathcal{C}$ that satisfies the following properties. \\ ~ \\
{\bf Assumptions on the class $\mathcal{C}$.}

$\mathcal{C}$ is a class of piecewise expanding maps on $[0,1]$ such that:

\begin{enumerate}

\item $\mathcal{C}$  satisfies $(D_r)$ for some $r \ge 1$ and {\em (Dec)}. 

\item There exist $\lambda > 1$ and $M \ge 1$ such that $\inf |T'| \ge \lambda$ and $\sup |T''|\le M$ for all $T  \in \mathcal{C}$.

\item There exists $N \ge 1$ such that $\sharp \mathcal{A}_T \le N$ for all $T  \in \mathcal{C}$.

\end{enumerate}

 ~ \\ The main result is:

\begin{The} \label{th:main} \em Let $(T_n)_{n \ge 1}$ be a sequence in $\mathcal{C}$ verifying the condition {\bf (Min)}. Then the process $Z_n = T_1^n$ satisfies an exponential concentration inequality.
\end{The}

In order to prove the theorem, we will use the McDiarmid's bounded difference method \cite{McDi1, McDi2}, already used in \cite{CG} and \cite{CMS}, that we will adapt to the non-stationary case.

We first extend the function $K$ as a function $K : X^{\mathbb{N}} \to \mathbb{R}$ that depends only on the $n$ first coordinates. $X^{\mathbb{N}}$ is endowed with the product $\sigma$-algebra $\tilde{\mathcal{F}}$ and the probability measure $\tilde{m}$ which is the image of $m$ by the map $\Phi : X \to X^{\mathbb{N}}$ defined by $\Phi(x) = (T_1^nx)_{n \ge 0 }$, with the convention $T_1^0 = {\rm Id}$. With this notation, $\mathbb{E}_m(e^{K(Z_0, \ldots, Z_{n-1}) - \mathbb{E}_m(K(Z_0, \ldots, Z_{n-1})}) = \mathbb{E}_{\tilde{m}}(e^{K - \mathbb{E}_{\tilde{m}}(K)})$. Let $\tilde{\mathcal{F}}_p$ be the $\sigma$-algebra on $X^{\mathbb{N}}$ of events depending only on the coordinates $(x_k)_{k \ge p}$, and define $K_p = \mathbb{E}_{\tilde{m}}(K| \tilde{\mathcal{F}}_p)$. Since the sequential dynamical system $(T_n)$ is exact, as a consequence of the assumption {\em (Dec)}, the $\sigma$-algebra $\mathcal{F}_{\infty}$ is trivial mod $m$. Consequently, the $\sigma$-algebra $\tilde{\mathcal{F}}_{\infty} := \cap_{p \ge 0}
 \tilde{\mathcal{F}}_p$ is trivial mod $\tilde{m}$, since it is easy to check that $\Phi^{-1}(\tilde{\mathcal{F}}_{\infty}) \subset \mathcal{F}_{\infty}$. Hence, by Doob's convergence theorem for martingales, $K_p$ goes $\tilde{m}$-a.s. to $\mathbb{E}_{\tilde{m}}(K)$, and in all $L^q$, with $1 \le q < \infty$. We then have $ K - \mathbb{E}_{\tilde{m}}(K) = \sum_{p\ge 0} D_p$, where $D_p = K_p - K_{p+1}$. Using Azuma-Hoeffding inequality, see \cite{Az67}, we deduce the existence of an universal  constant $C > 0$ such that for all $P \ge 0$, $$\mathbb{E}_{\tilde{m}}(e^{\sum_{p=0}^P D_p}) \le e^{C \sum_{p=0}^P {\rm sup} |D_p|^2}.$$

It remains to bound $D_p$: 

\begin{Pro} \label{prop:dpbound} \em There exist $\rho < 1$ and $C > 0$ depending only on $(T_n)_{n \ge 1}$ such that for all $p$ and all $K$, one has $$ |D_p| \le  C \sum_{j=0}^p \rho^{p-j} \lip_j(K).$$
\end{Pro}

This proposition, together with the Schwarz inequality, implies the desired concentration inequality, in the same manner as in \cite{CG}: indeed, we have $$\left(\sum_{j=0}^p \rho^{p-j} \lip_j(K) \right)^2 \le \left( \sum_{j=0}^p \rho^{p-j} \lip_j^2(K) \right) \left( \sum_{j=0}^p \rho^{p-j} \right) \le C \sum_{j=0}^p \rho^{p-j} \lip_j^2(K).$$ After summation over $p$, we get $\sum_{p=0}^P {\rm sup} |D_p|^2 \le C \sum_j \lip_j^2(K)$.

Proposition \ref{prop:dpbound} will follow immediately from the Lipschitz condition on $K$ and the following lemma:

\begin{Lem} \label{lem:eq} \em There exist $ \rho < 1$ and $C > 0$ depending only on $(T_n)_{n \ge 1}$, such that for all $K$, all $p$ and $\tilde{m}$-a.e. $\underline{x} = (x_0, x_1, \ldots ) \in X^{\mathbb{N}}$, one has $$\left| K_p(\underline{x}) - \int_X K(y, T_1 y, \ldots, T_1^{p-1} y, x_p, \ldots) \, dm(y) \right| \le C \sum_{j=0}^{p-1} \lip_j(K) \rho^{p-j}.$$
\end{Lem}
The rest of this section is devoted to the proof of this lemma. 

Firstly, we remark that $$K_p(\underline{x}) = \frac{1}{P_1^p \mathds{1}(x_p)} \sum_{T_1^p y = x_p}  \frac{K(y, T_1 y, \ldots, T_1^{p-1} y, x_p, \ldots )}{| (T_1^p)'(y)|}.$$
The proof of this identity can be found in Appendix \ref{app:formula}.

We fix a $x_{\star} \in X$ and we decompose $K_p$ as $$K_p(\underline{x}) = K(x_{\star}, \ldots, x_{\star}, x_p, \ldots) + \frac{1}{P_1^p \mathds{1}(x_p)}\sum_{i=0}^{p-1} \sum_{T_1^p y = x_p} \frac{H_i(y)}{|(T_1^p)'(y)|},$$
where $H_i(y) = K(y, \ldots, T_1^i y, x_{\star}, \ldots, x_{\star}, x_p, \ldots) - K(y, \ldots, T_1^{i-1} y, x_{\star}, \ldots, x_{\star}, x_p, \ldots)$.

Using the chain rule, we obtain $K_p(\underline{x}) = K(x_{\star}, \ldots, x_{\star}, x_p, \ldots) + \frac{1}{P_1^p \mathds{1}(x_p)}\sum_{i=0}^{p-1} P_{i+1}^p f_i(x_p)$, with $$f_i(y) = \sum_{T_1^i z = y} \frac{H_i(z)}{|(T_1^i)'(z)|} = P_1^i H_i(y).$$

Remark that $\int f_i \, dm = \int H_i \, dm$, whence 
$$\sum_{i=0}^{p-1} \int f_i \, dm = \int K(y, \ldots, T_1^{p-1} y, x_p, \ldots) \, dm(y) - K(x_{\star}, \ldots, x_{\star}, x_p, \ldots).$$
Thus,
\begin{eqnarray*} 
\lefteqn{K_p(\underline{x}) - \int  K(y, T_1 y, \ldots, T_1^{p-1} y, x_p, \ldots) \, dm(y)} \\ &=& \sum_{i=0}^{p-1} \left( \frac{P_{i+1}^p \mathds{1}(x_p)}{P_1^p \mathds{1}(x_p)} - 1 \right) \int f_i \, dm + \frac{1}{P_1^p \mathds{1}(x_p)} \sum_{i=0}^{p-1}  \left( P_{i+1}^p f_i (x_p) - \left( \int f_i \, dm \right) P_{i+1}^p \mathds{1}(x_p) \right).
\end{eqnarray*}

For the first term, we have
$$\left| \frac{P_{i+1}^p \mathds{1}(x_p)}{P_1^p \mathds{1}(x_p)} - 1 \right| = \frac{\left| P_{i+1}^p \mathds{1}(x_p) - P_1^p \mathds{1}(x_p)\right|}{P_1^p \mathds{1}(x_p)} \le C \| P_{i+1}^p( P_1^i \mathds{1} - \mathds{1} )\|_{\rm BV},$$ 
where we used {\bf (Min)} and the fact that the {\rm BV} norm dominates the supremum norm. Moreover, from {\em (Dec)}, we get $\| P_{i+1}^p( P_1^i \mathds{1} - \mathds{1} )\|_{\rm BV} \le K \theta^{p-i-1} \|P_1^i \mathds{1} - \mathds{1} \|_{\rm BV} \le C \theta^{p-i-1}$ since $(P_1^i \mathds{1})_i$ is bounded in {\rm BV} by (H3). Finally, since $\left| \int f_i \, dm \right| \le \lip_i(K)$, we obtain
\begin{equation}\label{eqsumpmun}
\left| \sum_{i=0}^{p-1} \left( \frac{P_{i+1}^p \mathds{1}(x_p)}{P_1^p \mathds{1}(x_p)} - 1 \right) \int f_i \, dm \right| \le C \sum_{i=0}^{p-1} \theta^{p-i-1} \lip_i(K).
\end{equation}

For the second term, the hypothesis {\em (Dec)} allows us to estimate 
\begin{equation}\label{fimintfi}
\left\|P_{i+1}^p f_i - \left( \int f_i \, dm \right) P_{i+1}^p \mathds{1} \right\|_{\rm BV} \le K \theta^{p- i -1} \left\|f_i - \int f_i \, dm \right\|_{\rm BV} \le C \theta^{p-i-1} \|f_i\|_{\rm BV}.
\end{equation}

Using {\bf (Min)}, \eqref{fimintfi} and $\| . \|_{\rm sup} \le \| . \|_{\rm BV}$, we have 
\begin{equation}\label{eqsumfimint}
\left| \frac{1}{P_1^p \mathds{1}(x_p)} \sum_{i=0}^{p-1}  \left( P_{i+1}^p f_i (x_p) - \left( \int f_i \, dm \right) P_{i+1}^p \mathds{1}(x_p) \right) \right| \le C \sum_{i=0}^{p-1} \theta^{p-i-1} \| f_i\|_{\rm BV}.
\end{equation}

Thus, it remains to estimate $\|f_i\|_{\rm BV}$. One can observe that $\|f_i\|_{\rm BV} \le \|f_i\|_{\rm sup} + {\rm V}(f_i)$ and that for any $y$, we have 
\begin{equation}\label{eqfisup}
|f_i(y)| \le P_1^i \mathds{1}(y) \|H_i\|_{\rm sup} \le P_1^i \mathds{1}(y) \lip_i(K) \le C \lip_i(K),
\end{equation}
since $(P_1^i \mathds{1})_i$ is bounded in {\rm BV}. The crucial point hence lies in the estimate of the variation of $f_i$.

To do so we first establish a distortion control. The proof is standard, but we reproduce it here for completeness. If $T$ is a piecewise $C^2$ map of the interval, with partition of monotonicity $\mathcal{A}_T$, we define its distortion ${\rm Dist}(T)$ as the least constant $C$ such that $|T'(x) - T'(y)| \le C |T'(x)| |Tx - Ty|$ for all $x,y \in I$ and $I \in \mathcal{A}_T$. Note that $(1+C)^{-1}\le \frac{ |T'(x)| }{|T'(y)|} \le (1+C)$ for all $x,y \in I$.

\begin{Lem} \label{lem:dist} \em There exists $C > 0$ such that ${\rm Dist}(T_1^n) \le C$ for all $n$.
\end{Lem}

\begin{proof}

As previously, we denote by $\mathcal{A}_n$ the partition of monotonicity of $T_n$, and by $\mathcal{A}_n^m$ the partition of $T_n^m$, $n \le m$. We have for any $x,y \in I \in \mathcal{A}_1^n$: $$\left| \log \frac{|(T_1^n)'(x)|}{|(T_1^n)'(y)|}\right| \le \sum_{j=1}^n \left| \log |T_j'(T_1^{j-1} x) | - \log |T_j'(T_1^{j-1}y)|\right| \le C \sum_{j=1}^n |T_1^{j-1} x - T_1^{j-1} y|,$$ where we have used the fact that the Lipschitz constant of $\log |T_j'|$ is bounded independently of $j$, since $|T_j''| \le M$. As all maps are uniformly expanding by a factor at least $\lambda$, and $x, y$ belong to the same partition element, we have $|T_1^{j-1} x - T_1^{j-1}y| \le \lambda^{-(n-j-1)} |T_1^n x - T_1^n y|$. Since $\sum_j \lambda^{-(n-j-1)}$ is bounded, this concludes the proof.
\end{proof}

We will need the following technical lemma, which is adapted from Lemma II.4 in \cite{CMS}.

\begin{Lem} \label{lem:CMS} \em There exists a constant $C > 0$ such that for all $n \ge 1$, $$\sum_{I \in \mathcal{A}_1^n} \sup_I \frac{1}{|(T_1^n)'|} \le C.$$
\end{Lem}

\begin{proof}
For any $I \in \mathcal{A}_1^n$, there exists a least integer $p = p_I \le n-1$ such that $T_1^p(I) \cap \partial \mathcal{A}_{p+1} \neq \emptyset$.  Denote by $\mathcal{A}_1^{n,p}$ the set of all $I \in \mathcal{A}_1^n$ for which $p_I = p$. Consider $I \in \mathcal{A}_1^{n,p}$. There exists $a \in \partial I$ such that $b = T_1^p a \in \partial \mathcal{A}_{p+1}$. From Lemma \ref{lem:dist}, we deduce that there exists $C > 0$ such that for any $x \in I$, $$|(T_1^n)'(x)| \ge C |(T_1^n)'(a)| = C |(T_{p+1}^n)'(T_1^p a)| | (T_1^p)'(a)| \ge C \lambda^{n-p} |(T_1^p)'(a)|.$$

Hence, one has $\sup_I \frac{1}{|(T_1^n)'|} \le C^{-1} \lambda^{-(n-p)} \frac{1}{|(T_1^p)'(a)|}$. Since a pre-image by $T_1^p$ of an element $b \in \mathcal{A}_{p+1}$ can only belong to at most two different $I  \in \mathcal{A}_1^n$, it follows that 

\begin{eqnarray*}
\sum_{I \in \mathcal{A}_1^n} \sup_I \frac{1}{|(T_1^n)'|} &\le& 2 C^{-1} \sum_{p=0}^{n-1} \lambda^{-(n-p)} \sum_{b \in \partial \mathcal{A}_{p+1}} \sum_{T_1^p a = b} \frac{1}{|(T_1^p)'(a)|} \\ &=&  2 C^{-1} \sum_{p=0}^{n-1} \lambda^{-(n-p)} \sum_{b \in \partial \mathcal{A}_{p+1}} P_1^p \mathds{1} (b)  \le 2 C^{-1} \sum_{p=0}^{n-1} \lambda^{-(n-p)} \|P_1^p \mathds{1}\|_{\rm BV} \sharp \partial \mathcal{A}_{p+1}.
\end{eqnarray*}

This quantity is bounded independently of $n$, since $(P_1^p)_p$ is bounded in ${\rm BV}$, and the number of elements in $\mathcal{A}_{p+1}$ is bounded independently of $p$, by assumption.
\end{proof}

As a direct corollary of the two previous lemma, we have 

\begin{Cor} \label{cor:var} \em There exists $C > 0$ such that for all $n \ge 1$, $$\sum_{I \in \mathcal{A}_1^n} {\rm V}_I \left( \frac{1}{|(T_1^n)'|}\right) \le C,$$ where ${\rm V}_I(f)$ denotes the total variation of $f$ over the subinterval $I \subset [0,1]$.
\end{Cor}

\begin{proof}
Let $x_0 < \ldots < x_j$ be a sequence of elements of $I$. Then 

\begin{eqnarray*} \sum_{i=0}^{j-1} \left| \frac{1}{|(T_1^n)'(x_{i+1})|} - \frac{1}{|(T_1^n)'(x_i)}\right| &=& \sum_{i=0}^{j-1} \frac{\big\vert \left|(T_1^n)'(x_{i+1})\right| - \left|(T_1^n)'(x_i)\right|\big\vert}{|(T_1^n)'(x_i)| |(T_1^n)'(x_{i+1})|} \\ &\le& \sum_{i=0}^{j-1} \frac{{\rm Dist}(T_1^n) |T_1^n x_{i+1} - T_1^n x_i|}{|(T_1^n)'(x_i)|} \le C \sup_I \frac{1}{|(T_1^n)'|},
\end{eqnarray*} 
where for the last inequality, we have used Lemma \ref{lem:dist} and the fact that $(T_1^n x_i)_i$ is a monotone sequence in $[0,1]$. A direct application of Lemma \ref{lem:CMS} proves the corollary.\end{proof}

We are now able to estimate ${\rm V}(f_i)$. For $I \in \mathcal{A}_1^i$, we denote by $S_{i,I}$ the inverse branch of the restriction of $T_1^i$ to $I$. Then, we can write $$f_i = P_1^i H_i = \sum_{I \in \mathcal{A}_1^i} \left( \frac{H_i}{|(T_1^i)'|} \right) \circ S_{i, I} \, \mathds{1}_{T_1^i (I)}.$$ Using standard properties of the total variation, it follows that $${\rm V}(f_i) \le \sum_{I \in \mathcal{A}_1^i} {\rm V}_I \left( \frac{H_i}{|(T_1^i)'|} \right) + 2 \sum_{a \in \partial \mathcal{A}_1^i} \frac{|H_i(a)|}{|(T_1^i)'(a)|} \le {\rm I}_i + {\rm II}_i + {\rm III}_i,$$ where 

\[\begin{aligned}
&{\rm I}_i& &=& &\sum_{I \in \mathcal{A}_1^i} {\rm V}_I \left(\frac{1}{|(T_1^i)'|} \right) \sup_I |H_i|,&
\\
&{\rm II}_i& &=& &\sum_{I \in \mathcal{A}_1^i} \sup_I \frac{1}{|(T_1^i)'|} {\rm V}_I (H_i),&
\\
&{\rm III}_i& &=& &2  \sum_{a \in \partial \mathcal{A}_1^i} \frac{|H_i(a)|}{|(T_1^i)'(a)|}.&
\end{aligned}
\]

Using the Lipschitz condition on $K$, one gets ${\rm I}_i \le C \lip_i(K) \sum_{I \in \mathcal{A}_1^i} {\rm V}_I \left(\frac{1}{|(T_1^i)'|} \right)$, which gives ${\rm I}_i \le C \lip_i(K)$ by Corollary \ref{cor:var}. 

Let us now estimate ${\rm II}_i$. Let $y_0 < \ldots < y_l$ be a sequence of points in $I$. In order to estimate $\sum_{j=0}^{l-1} \left| H_i(y_{j+1}) - H_i(y_j) \right|$, we split $H_i$ in two terms in an obvious way, and we deal with the first one, the second being completely similar. We have $$\begin{aligned} \sum_{j=0}^{l-1} \sum_{k=0}^i \left| K(y_{j+1}, \ldots, T_1^k y_{j+1}, T_1^{k+1} y_j, \ldots, T_1^i y_j, \ldots) - K(y_{j+1}, \ldots, T_1^{k-1} y_{j+1}, T_1^k y_j, \ldots, T_1^i y_j, \ldots) \right|\\  \le \sum_{j=0}^{l-1} \sum_{k=0}^i \lip_k(K) |T_1^k y_{j+1} - T_1^ky_j| \le \sum_{k=0}^i \lip_k(K) m(T_1^k(I)). \end{aligned} $$

Since $I$ belongs to $\mathcal{A}_1^i$, the interval $T_1^k(I)$ is included in an interval of monotonicity of $T_{k+1}^i$, and hence its length is less than $\lambda^{-(i-k)}$. Therefore, one has ${\rm V}_I(H_i) \le 2\sum_{k=0}^i \lambda^{-(i-k)} \lip_k(K)$. An application of Lemma \ref{lem:CMS} then yields ${\rm II}_i \le C \sum_{k=0}^i \lambda^{-(i-k)} \lip_k(K)$.

Using again Lemma \ref{lem:CMS}, we can bound the third term by ${\rm III}_i \le C \lip_i(K)$. 

Putting together all the estimates, we find that 
\begin{equation}\label{eqestvfi}
{\rm V}(f_i) \le C \sum_{k=0}^i \lambda^{-(i-k)} \lip_k(K).
\end{equation}
Finally, \eqref{eqfisup} and \eqref{eqestvfi} give the same estimate for $\|f_i\|_{\rm BV}$.

Coming back to $K_p(\underline{x})$, we then have by \eqref{eqsumpmun} and \eqref{eqsumfimint}:

\[
\left| K_p(\underline{x}) - \int  K(y, T_1 y, \ldots, T_1^{p-1} y, x_p, \ldots) \, dm(y) \right| \le C \sum_{i=0}^{p-1} \theta^{p-i-1} \sum_{k=0}^i \lambda^{-(i-k)} \lip_k(K),
\]

which is less than $C \sum_{k=0}^{p-1} \rho^{p-k} \lip_k(K)$, as shown by a simple computation. This concludes the proof of our main theorem.

\section{Applications} \label{section:app}

Let $(T_n)$ be a sequential dynamical systems on $[0,1]$ that satisfies all the conditions of Theorem \ref{th:main}. We describe some applications of the concentration inequality, following ideas from \cite{CCSII, CMS}.

\subsection{Large deviations for ergodic sums}

Considering observables $K$ of the form $$K(x_0, \ldots, x_{n-1}) = \sum_{k=0}^{n-1} f(x_k)$$ with $f$ Lipschitz, an immediate application of Theorem \ref{th:main} gives us:

\begin{Pro} \label{pro:ld_lip} \em There exists $C > 0$ such that for all $f : [0,1] \to \mathbb{R}$ Lipschitz, all $n \ge 1$ and all $t > 0$:

\[ m\left(\frac{1}{n} \sum_{k=0}^{n-1} \left[f \circ T_1^k - \int f \circ T_1^k \, dm \right] > t\right) \le e^{- \frac{C n t^2}{ \lip(f)^2}}.\]
\end{Pro}

In Section~\ref{secaltappr}, we prove a similar statement for observables $f$ in BV.

\subsection{Empirical measure}

For $x \in [0,1]$, define for $n \ge 1$ the empirical measure $$\mathcal{E}_n(x) = \frac{1}{n}\sum_{k=0}^{n-1} \delta_{T_1^k x}.$$ By the strong law of large numbers, we know that for a.e. $x$, this measure approximates in the weak topology, as $n \to \infty$, the measure $m_n$ defined by $$m_n = \frac{1}{n} \sum_{k=0}^{n-1} (T_1^k)_{\star}m.$$
It is natural to try to quantify this phenomenon. For this, we need first to introduce a notion of distance on the set of probability measures. We will consider the Kantorovich distance, which, for probability measures $\mu_1$ and $\mu_2$ on $[0,1]$, can be defined as $$\kappa(\mu_1, \mu_2) = \int_0^1 |F_{\mu_1}(t) - F_{\mu_2}(t)| \, dt,$$ where $F_{\mu_i}(t) = \mu_i([0,t])$ is the distribution function of $\mu_i$. Thus, we have:

\begin{Pro}\label{prop:emp} \em There exist $t_0 > 0$ and $C > 0$ such that for all $t > t_0$ and $n \ge 1$:

\[m\left( \kappa(\mathcal{E}_n, m_n) > \frac{t}{\sqrt{n}} \right) \le e^{-Ct^2}. \]
\end{Pro}
The following proof is an application of our main theorem, in Section~\ref{secaltappr} we prove this proposition using a direct approach.
\begin{proof} For $n \ge 1$, define $$K_n(x_0, \ldots, x_{n-1}) = \int_0^1 \left| \frac{1}{n} \sum_{k=0}^{n-1} \mathds{1}_{[0,t]}(x_k) - F_{m_n}(t) \right| \, dt.$$
We clearly have $\kappa(\mathcal{E}_n(x), m_n) = K_n(x, \ldots, T_1^{n-1} x)$ and $\lip_j(K_n) \le \frac{1}{n}$ for any $0 \le j \le n-1$. By the concentration inequality, and its large deviation counterpart, we derive:

$$m \left(\kappa(\mathcal{E}_n, m_n) - \mathbb{E}_m(\kappa(\mathcal{E}_n, m_n)) > \frac{t}{\sqrt{n}} \right) \le e^{-Ct^2}.$$
To conclude, it is then sufficient to prove that $\mathbb{E}_m(\kappa(\mathcal{E}_n, m_n))$ is of order $\frac{1}{\sqrt{n}}$.

Denote by $\chi_t$ the characteristic function $\mathds{1}_{[0,t]}$. Using Schwarz inequality, we have 

$$ \begin{aligned} 
\mathbb{E}_m(\kappa(\mathcal{E}_n, m_n)) = \int \left( \int_0^1 \left| \frac{1}{n} \sum_{k=0}^{n-1} \chi_t(T_1^k x) - F_{m_n}(t) \right| \, dt \right) \, dm(x)
\\ \le \left[ \int_0^1 \left( \int \left| \frac{1}{n} \sum_{k=0}^{n-1} \left(  \chi_t \circ T_1^k - F_{(T_1^k)_{\star}m}(t) \right) \right|^2 \, dm \right) \, dt \right]^{\frac{1}{2}}.
\end{aligned}$$

Expanding the square, we have \begin{eqnarray*} \lefteqn{\int \left| \frac{1}{n} \sum_{k=0}^{n-1} \left( \chi_t \circ T_1^k - F_{(T_1^k)_{\star}m}(t) \right) \right|^2  dm} \\  &=& \frac{1}{n^2} \sum_{k,l=0}^{n-1} \int \left( \chi_t \circ T_1^k - F_{(T_1^k)_{\star}m} (t) \right) \left( \chi_t \circ T_1^l - F_{(T_1^l)_{\star}m} (t) \right)  dm. \end{eqnarray*}

If $k \le l$, by the properties of transfer operators, and since $F_{(T_1^k)_{\star}m}(t) = \int \chi_t \circ T_1^k \, dm$, \begin{eqnarray*} \lefteqn{\int \left( \chi_t \circ T_1^k - F_{(T_1^k)_{\star}m} (t) \right) \left( \chi_t \circ T_1^l - F_{(T_1^l)_{\star}m} (t) \right)  dm} \\ &=& \int \left( \chi_t - \int \chi_t \circ T_1^l \right) P_{k+1}^l \left( \left( \chi_t - \int \chi_t \circ T_1^k \right) P_1^k \mathds{1} \right) dm . \end{eqnarray*}

Since $\left( \chi_t - \int \chi_t \circ T_1^k \right) P_1^k \mathds{1}$ has $0$ integral, $\{\chi_t\}_{t}$ is a bounded family in BV, and $(P_1^k \mathds{1})_k$ is also bounded in BV, we can use property {\em (Dec)} to get 

\begin{equation}\label{eqdecemp}
 \left| \int \left( \chi_t - \int \chi_t \circ T_1^l \right) P_{k+1}^l \left( \left( \chi_t - \int \chi_t \circ T_1^k \right) P_1^k \mathds{1} \right) dm\right| \le C \theta^{l-k}. 
 \end{equation}

It follows that $$\mathbb{E}_m(\kappa(\mathcal{E}_n, m_n)) \le \frac{C}{n} \left( \sum_{k,l=0}^{n-1} \theta^{- |k-l|} \right)^{\frac{1}{2}} = \mathcal{O}\left( \frac{1}{\sqrt{n}}\right).$$\end{proof}

\subsection{Efficiency of shadowing}

Let $A$ be a measurable subset of $[0,1]$ with positive measure. How well can we approximate the trajectory of a point $x \in [0,1]$ by a trajectory starting in $A$ ? The following result provides an estimation of the average quality of this shadowing.

\begin{Pro} \label{pro:shadow} \em There exist constants $C_1, C_2 > 0$ such that for all $A$ with $m(A) > 0$ and all $n \ge 1$  the sequence $(Z_n)$ of functions defined by 
$$Z_n(x) = \inf_{y \in A} \frac{1}{n}\sum_{k=0}^{n-1} |T_1^k x - T_1^k y|,$$
satisfies for any $t > 0$:
$$m\left( Z_n > C_1 \frac{\sqrt{\left| \log m(A) \right|}}{\sqrt{n}} + \frac{t}{\sqrt{n}}\right) \le e^{-C_2 t^2}.$$
\end{Pro}

\begin{proof} We apply the concentration inequality to the functions $$K_n(x_0, \ldots, x_{n-1}) = \frac{1}{n} \inf_{y \in A} \sum_{k=0}^{n-1} |T_1^k y - x_k|.$$ They satisfy $\lip_k(K_n) \le \frac{1}{n}$, so we get $$m \left( Z_n - \mathbb{E}_m(Z_n) > \frac{t}{\sqrt{n}} \right) \le e^{-Ct^2}.$$

We estimate now $\mathbb{E}_m(Z_n)$. One has $$m(A) \le m \left( \mathbb{E}_m(Z_n) - Z_n  > \frac{1}{\sqrt n} \frac{\sqrt n | \mathbb{E}_m(Z_n)|}{2} \right),$$
since $Z_n = 0$ on $A$. By the concentration inequality, we deduce that $m(A) \le e^{-C n ( \mathbb{E}_m ( Z_n))^2/4}$, whence $$\mathbb{E}_m({Z}_n) \le C_1 \frac{\sqrt{\left|\log m(A)\right|}}{\sqrt n}.$$
\end{proof}

\subsection{Almost sure central limit theorem} Let $(Z_n)$ be a random process that satisfies the central limit theorem, i.e. there exist $A_n \in \mathbb{R}$ and $B_n > 0$ such that $S_n = \sum_{k=0}^{n-1} Z_k$ verifies $$\frac{S_n - A_n}{B_n} \to W$$ in the weak topology, where $W$ is distributed as a standard normal law $\mathcal{N}(0,1)$.

By definition, this means that for every $t \in \mathbb{R}$ $$\mathbb{E}\left( \mathds{1}_{ \left\{ \frac{ S_n - A_n}{B_n} \le t \right\} } \right) \to \mathbb{P}(W \le t) = \frac{1}{\sqrt{2 \pi}} \int_{- \infty}^t e ^{- \frac{s^2}{2}} \, ds.$$

We can then ask about the pointwise behavior of $ \mathds{1}_{ \left\{ \frac{ S_n - A_n}{B_n} \le t \right\}}$. If this quantity cannot in general converge almost surely, it appears that after a correct averaging, one can recover the convergence. It turns out that the correct averaging is a logarithmic one, more precisely, we say that $(Z_n)$ satisfies an almost sure central limit theorem if almost surely $$\frac{1}{H_n} \sum_{k=1}^n  \frac{1}{k} \, \delta_{\frac{S_k - A_k}{B_k}} \to W$$ in the weak topology, where $H_n = \sum_{k=1}^n \frac{1}{k}$ (one can observe that $H_n\sim \log n$).

We defer the interested reader to \cite{Ber98, Jon} for particularly nice reviews of this subject, and to \cite{CG07} for a description of applications in dynamical systems.

Let now $(T_n)$ be a sequential dynamical systems on $[0,1]$ which satisfies the same assumptions as before. Let $f : [0,1] \to \mathbb{R}$ be a Lipschitz observable, and define $$S_n = \sum_{k=0}^{n-1} \left[ f \circ T_1^k - \int f \circ T_1^k \, dm \right].$$ If $\|S_n\|_2$ is unbounded, we have by Conze and Raugi that $\frac{S_n}{\|S_n\|_2}$ goes in distribution to $W \sim \mathcal{N}(0,1)$.

In order to prove an almost-sure version of this convergence, we will need an additional assumption on the growth of the variance of the ergodic sum, namely that there exists a $C > 0$ such that $\|S_n\|_2 \ge C n^{\frac{1}{2}}$ for all $n$.

\begin{Pro} \em Let $f : [0,1] \to \mathbb{R}$ be a Lipschitz observable, and define $S_n$ as above. Assume that $\|S_n\|_2 \ge C n^{\frac{1}{2}}$ for some $C >0$. Then $S_n$ satisfies the almost sure central limit theorem:  
$$\frac{1}{H_n} \sum_{k=1}^n \frac{1}{k} \, \delta_{\frac{S_k(x)}{\|S_k\|_2}} \overset{\mathcal{L}}{\rightarrow} W$$
for almost every $x \in [0,1]$.
\end{Pro}

\begin{proof} We have to show that almost everywhere, $\frac{1}{H_n} \sum_{k=1}^n \frac{1}{k} \, g \left( \frac{S_k}{\|S_k\|_2} \right) \to \mathbb{E}(g(W))$ for all $g : \mathbb{R} \to \mathbb{R}$ continuous and bounded. By standard density and separability arguments (see e.g. Theorem 2.4 in \cite{Jon}), it is sufficient to prove that for all $g$ bounded and Lipschitz, this convergence occurs almost surely. Remark also that we can assume without loss of generality that $g(0) = 0$. Hence there exists $C>0$ such that $ |g(x)| \le C |x|$ for all $x$.

Define $$K_n(x_0, \ldots, x_{n-1}) = \frac{1}{H_n} \sum_{k=1}^n \frac{1}{k} \, g\left( \frac{\sum_{j=0}^{k-1} \left[f(x_j) - \int f \circ T_1^j \right]}{\|S_k\|_2} \right).$$
We have then $K_n(x, \ldots, T_1^{n-1}x) = \frac{1}{H_n} \sum_{k=1}^n \frac{1}{k} \, g \left( \frac{S_k}{\|S_k\|_2} \right)$, and the assumption on $\|S_n\|_2$ gives
$$\lip_j(K_n) \le \frac{1}{H_n} \sum_{k > j} \frac{1}{k} \lip(g) \frac{\lip(f)}{\|S_k\|_2} = \mathcal{O}\left(\frac{1}{H_n} \sum_{k > j} k^{- \frac{3}{2}} \right) = \mathcal{O}\left(\frac{1}{j^{ \frac{1}{2}} H_n}\right).$$

Hence $$\sum_{j=0}^{n-1} \lip_j^2(K_n) = \mathcal{O} \left( \frac{1}{H_n^2} \sum_{j=0}^{n-1} j^{-1}  \right) = \mathcal{O}\left( \frac{1}{\log n} \right).$$

Using the concentration inequality, we deduce that 

\[ m\left( \left|\frac{1}{H_n} \sum_{k=1}^n \frac{1}{k} \, g \left( \frac{S_k}{\|S_k\|_2} \right) - \mathbb{E}_m\left( \frac{1}{H_n} \sum_{k=1}^n \frac{1}{k} \, g \left( \frac{S_k}{\|S_k\|_2} \right)\right) \right| \ge t \right) \le e^{-C t^2 \log n}. \]

Hence, if we define the subsequence $n_j = e^{j^{\alpha}}$ with $\alpha > 0$, we have that $\frac{1}{H_{n_j}} \sum_{k=1}^{n_j} \frac{1}{k} \, g \left( \frac{S_k}{\|S_k\|_2} \right)$ goes to $\mathbb{E}(g(W))$ by the Borel-Cantelli lemma, since $\mathbb{E}_m \left(\frac{1}{H_{n_j}} \sum_{k=1}^{n_j} \frac{1}{k} \, g \left( \frac{S_k}{\|S_k\|_2} \right) \right)$ goes to $\mathbb{E}(g(W))$ when $j \to \infty$ by the usual CLT.

It remains to control the gaps. If $n_j < n \le n_{j+1}$, we have 

\begin{eqnarray*}
\lefteqn{\frac{1}{H_n} \sum_{k=1}^n \frac{1}{k} \, g \left( \frac{S_k}{\|S_k\|_2} \right) - \frac{1}{H_{n_j}} \sum_{k=1}^{n_j} \frac{1}{k} \, g \left( \frac{S_k}{\|S_k\|_2} \right)} \\ &=& \frac{1}{H_n} \sum_{k=n_j +1}^n \frac{1}{k} \, g \left( \frac{S_k}{\|S_k\|_2} \right) + \left( \frac{H_{n_j}}{H_n} - 1 \right) \frac{1}{H_{n_j}} \sum_{k=1}^{n_j} \frac{1}{k} \, g \left( \frac{S_k}{\|S_k\|_2} \right).
\end{eqnarray*}

Using the previous observations, the supremum over $n_j < n \le n_{j+1}$ of the second term tends to $0$ when $j \to \infty$. For the first term, the assumption on $g$ gives us
\[
\sup_{n_j < n \le n_{j+1}} \left| \frac{1}{H_n} \sum_{k=n_j +1}^n \frac{1}{k} \, g \left( \frac{S_k}{\|S_k\|_2} \right) \right| \le \frac{C}{H_{n_j}} \sum_{k=n_j +1}^{n_{j+1}} \frac{|S_k|}{k \|S_k\|_2}.
\]

We now prove the almost-sure convergence to zero of $G_j = \frac{1}{H_{n_j}} \sum_{k=n_j+1}^{n_{j+1}} \frac{|S_k|}{k \|S_k\|_2}$. Using the Schwarz inequality, we have 

\begin{eqnarray*} \mathbb{E}_m(G_j^2) \le \frac{1}{H_{n_j}^2} \sum_{k, l = n_j +1}^{n_{j+1}} \frac{\mathbb{E}_m(S_k^2)^{\frac 1 2} \mathbb{E}_m(S_l^2)^{\frac 1 2}}{k  \|S_k\|_2 l \|S_l\|_2} &=& \frac{1}{H_{n_j}^2} \left( \sum_{k = n_j+1}^{n_{j+1}} k^{-1} \right)^2 \\ &=& \frac{( \log n_{j+1} - \log n_j + \mathcal{O}(1))^2}{H_{n_j}^2} \le \frac{\mathcal{O}(1)}{j^2}. \end{eqnarray*}
The result follows by the Borel-Cantelli lemma, since $\mathbb{E}_m(G_j^2)$ is summable.\end{proof}

\section{Large deviations and empirical measure: an alternative approach}\label{secaltappr}
In the previous section, we gave some applications of Theorem~\ref{th:main}. In this section, we will show that one can use a direct approach to get large deviations estimates for observables in BV and also an upper bound on the Kantorovich distance.

\subsection{Large deviations estimates for observables in BV} \label{app:LD}

Proposition \ref{pro:ld_lip} gives exponential large deviations estimates for ergodic sums associated to Lipschitz observables, for sequential dynamical systems on $[0,1]$ that satisfy all the assumptions required by Theorem \ref{th:main}. As the strong law of large numbers holds for observables in BV under slighty weaker assumptions for the dynamical system, it is natural to ask whether large deviations estimates remain valid in this more general situation. The following proposition answers positively to this question:

\begin{Pro} \em Let $(T_n)$ be a sequential dynamical system on $[0,1]$ such that all maps $T_n$ belong to a class $\mathcal{C}$ that satisfies $(D_r)$ and {\em (Dec)}, and for which the condition {\bf (Min)} holds. Then, for any $f \in {\rm BV}$ and all $t > 0$, there exist $\tau = \tau(f) > 0$ and $C = C(f, t) > 0$ such that for all $n \ge 1$:

\[ m\left(\frac{1}{n} \sum_{k=0}^{n-1} \left[f \circ T_1^k - \int f \circ T_1^k \, dm \right] > t\right) \le C e^{- \tau n t^2}.\]

\end{Pro}

\begin{proof}
We follow the strategy employed in \cite{AFLV} Proposition 2.5, using the martingale approximation described in Section 5 of \cite{Conze_Raugi}. Write $f_n = f - \int f \circ T_1^n \, dm$ and define the operators $Q_n$, for $n \ge 1$, by 
$$Q_n g = \frac{P_n(g P_1^{n-1} \mathds{1})}{P_1^n \mathds{1}}.$$ 
Let $h_n$ defined by the relation $h_{n+1} = Q_{n+1} f_n + Q_{n+1} h_n$, with $h_0=0$. We get $$h_n = \frac{1}{P_1^n \mathds{1}} \sum_{k=0}^{n-1} P_{k+1}^n( f_k P_1^k \mathds{1}).$$

Since BV is a Banach algebra and $(P_1^k \mathds{1})_k$ is bounded in BV, we get that $f_k P_1^k \mathds{1}$ is a bounded sequence in BV of functions with integral $0$. As a consequence of {\em (Dec)}, the BV norm of the term $P_{k+1}^n(f_k P_1^k \mathds{1})$ decays exponentially fast with $n-k$, which implies, thanks to condition {\bf (Min)}, that the sequence $(h_n)$ is bounded for the supremum norm.

Define $\varphi_n = f_n + h_n - h_{n+1} \circ T_{n+1}$ and $U_n = \varphi_n \circ T_1^n$. We have that $(U_n)$ is a sequence of reversed martingales for the filtration $(\mathcal{F}_n)$, and is bounded for the supremum norm. Denoting by $S_n = \sum_{k=0}^{n-1} f_k \circ T_1^k = \sum_{k=0}^{n-1} \left[f \circ T_1^k - \int f \circ T_1^k \, dm \right]$, we have $$S_n = \sum_{k=0}^{n-1} U_k + h_n \circ T_1^n.$$

Then $m(S_n > n t)= m(\sum_{k=0}^{n-1} U_k +h_n \circ T_1^n > nt) \le m(\sum_{k=0}^{n-1} U_k > \frac{nt}{2})$ since $(h_n)$ is bounded for the supremum norm and thus $h_n \circ T_1^n < \frac{nt}{2}$ for all $n$ large enough. On the other hand, using the Azuma-Hoeffding inequality, there exists a constant $\tau = \tau(f)$ such that $ m(\sum_{k=0}^{n-1} U_k > \frac{nt}{2}) \le e^{- \tau n t^2}$, since $(U_n)$ is also bounded for the supremum norm. This proves the proposition.
\end{proof}

\subsection{Empirical measure and Kantorovich distance}
In Proposition~\ref{prop:emp}, an upper bound for the Kantorovich distance was obtained using Theorem~\ref{th:main}. Following Dedecker and Merlev\`ede \cite{DM}, we will prove this result for a larger class of sequential dynamical systems via an alternative approach. 

Recall that for a sequential dynamical systems $(T_n)$ on $[0,1]$, we define the empirical measure by $\mathcal{E}_n(x) = \frac{1}{n} \sum_{k=0}^{n-1} \delta_{T_1^k x}$ and the measure $m_n$ by $m_n = \frac{1}{n} \sum_{k=0}^{n-1} (T_1^k)_{\star}m$.

\begin{Pro} \em Let $(T_n)$ be a sequential dynamical system on $[0,1]$ such that all maps $T_n$ belong to a class $\mathcal{C}$ that satisfies $(D_r)$ and {\em (Dec)}, and for which the condition {\bf (Min)} holds. Then, there exists $C >0$ such that for all $t >0$ and all $n\ge 1$: 
\[m\left( \kappa(\mathcal{E}_n, m_n) > \frac{t}{\sqrt{n}} \right) \le 2 e^{-Ct^2}. \]
\end{Pro}

\begin{proof}
Let
\[S_n(t)=\sum_{k=0}^{n-1} \mathds{1}_{Z_k\leq t}-\mathbb{E}(\mathds{1}_{Z_k\leq t}).\]
Let $\|.\|_{L^p_m}$ denotes the $L^p$-norm under $m$, one can observe that
\[\kappa(\mathcal{E}_n,m_n)=\frac{1}{n}\|S_n\|_{L^1_m}\leq\frac{1}{n}\|S_n\|_{L^2_m}.\]
Let $\mathcal{M}_n=\{\emptyset,[0,1]\}$ be the trivial $\sigma$-algebra and let $\mathcal{M}_k=\sigma(T_1^k)$ for $0\leq k\leq n-1$. Thus, using Yurinskii's idea \cite{Yu}, since $\E(S_n|\mathcal{M}_n)=\E(S_n)=0$,  we have
\[S_n=\sum_{i=0}^{n-1}\E(S_n|\mathcal{M}_i)-\E(S_n|\mathcal{M}_{i+1})=\sum_{i=0}^{n-1}d_{i,n}\]
where
\[d_{i,n}(t)=\sum_{k=0}^{i}\E(\mathds{1}_{Z_k\leq t}|\mathcal{M}_i)-\E(\mathds{1}_{Z_k\leq t}|\mathcal{M}_{i+1}).\]
Thus, using Theorem 3 of \cite{Pi}, an Azuma-type inequality in Hilbert spaces, we get
\[m\left( \kappa(\mathcal{E}_n, m_n) > \frac{t}{\sqrt{n}} \right)\leq m\left( \frac{1}{\sqrt{n}}\|S_n\|_{2,m} > t \right) \le 2e^{-\frac{nt^2}{2b_n^2}} \]
where
\begin{equation} \label{eqn:bn} b_n^2=\sum_{i=0}^{n-1}\left\| \|d_{i,n}\|_{L^2_m}\right\|_\infty^2. \end{equation}
We thus only need to bound $b_n^2$. We observe that 
\begin{equation} \label{eqn:din} \|d_{i,n}\|_{L^2_m}\leq \sum_{k=0}^{i}\left\|\E(\mathds{1}_{Z_k\leq .}|\mathcal{M}_i)-\E(\mathds{1}_{Z_k\leq .})\right\|_{L^2_m}+\sum_{k=0}^{i}\left\|\E(\mathds{1}_{Z_k\leq .}|\mathcal{M}_{i+1})-\E(\mathds{1}_{Z_k\leq .})\right\|_{L^2_m} .\end{equation}
Now, one can use the properties of the transfer operator to show that
\[\E(\mathds{1}_{Z_k\leq t}|\mathcal{M}_i)-\E(\mathds{1}_{Z_k\leq t})=\frac{P_{k+1}^i\left((\chi_t-\E(\chi_t(Z_k)))P_1^k(\mathds{1})\right)}{P_1^i(\mathds{1})}\circ T_1^i,\]
moreover, using \textbf{(Min)},
\[\left\|\E(\mathds{1}_{Z_k\leq .}|\mathcal{M}_i)-\E(\mathds{1}_{Z_k\leq .})\right\|_{L^2_m}\leq \frac{\sup_{t\in[0,1]}|P_{k+1}^i\left((\chi_t-\E(\chi_t(Z_k)))P_1^k(\mathds{1})\right)|}{\delta}.\]
Finally, using {\em (Dec)} as in \eqref{eqdecemp}, we obtain
 \[\left\|\E(\mathds{1}_{Z_k\leq .}|\mathcal{M}_i)-\E(\mathds{1}_{Z_k\leq .})\right\|_{L^2_m}\leq\frac{C\theta^{i-k}}{\delta}.\]
Since $\sum_{k=0}^{i} \theta^{i-k}$ is bounded uniformly in $i$, from \eqref{eqn:bn} and \eqref{eqn:din} we get
 \begin{eqnarray*}
 b_n^2&\leq&\sum_{i=0}^{n-1} \sum_{i=0}^{n-1}\left\| \|d_{i,n}\|_{L^2_m}\right\|_\infty^2 \le C n \end{eqnarray*}
 which gives us,
 \[m\left( \kappa(\mathcal{E}_n, m_n) > \frac{t}{\sqrt{n}} \right)\leq 2 \textrm{ exp}\left(-C t^2\right)\]
and concludes the proof.
\end{proof}

\appendix

\section{Proof of a formula on conditional expectations} \label{app:formula}

In this appendix, we prove the following relation: 

\begin{equation}\label{eqkp}
 K_p(\underline{x}) = \frac{1}{P_1^p \mathds{1}(x_p)} \sum_{T_1^p y = x_p}  \frac{K(y, T_1 y, \ldots, T_1^{p-1} y, x_p, \ldots )}{| (T_1^p)'(y)|}, 
 \end{equation}

that holds for $\tilde{m}$-a.e. $\underline{x} \in X^{\mathbb{N}}$.

First, we recall some basic results on transfer operators that will be needed in the proof. Let $(X, \mathcal{A}, \mu)$ and $(Y, \mathcal{B}, \nu)$ be two probability spaces and let $T : X \to Y$ be a measurable non-singular map. Its associated transfer operator $P_T : L^1(\mu) \to L^1(\nu)$ is defined as follows: for $f \in L^1(\mu)$, $P_T(f)$ is the Radon-Nykodim derivative $\frac{d( \lambda \circ T^{-1})}{d \nu}$, where $d \lambda = f d \mu$.

It satisfies the following properties:

\begin{Lem} \label{lem:transfer_prop} \em For all $f \in L^1(\mu)$ and $g \in L^1(\nu)$, we have:

\begin{enumerate}

\item $\int_X f \, g \circ T d \mu = \int_Y P_T(f) \, g \, d\nu,$

\item $P_T(g \circ T) = g P_T \mathds{1},$

\item $ \mathbb{E}_{\mu}(f | T ^{-1} \mathcal{B}) = \left(\frac{P_T(f)}{P_T \mathds{1}} \right) \circ T,$ where $\frac{P_T(f)}{P_T \mathds{1}}$ is defined as $0$ on the set $\{P_T \mathds{1} = 0\}$. In particular, $\left(\frac{P_T(f_1)}{P_T(f_2)}\right) \circ T = \frac{\mathbb{E}_\mu(f_1 | T^{-1} \mathcal{B})}{\mathbb{E}_\mu(f_2| T^{-1} \mathcal{B}})$ for all $f_1, f_2 \in L^1(\mu)$.

\end{enumerate}

\end{Lem}

We recall now our setting. Let $(T_n)_{n \ge 1}$ be a sequence of piecewise expanding maps of the unit interval $(X,\mathcal{F},m)$. We define $\Phi : X  \to X^{\mathbb{N}}$ by $\Phi(x) = (T_1^n x)_{n \ge 0}$. $P_1^p$ denotes the transfer operator of the map $T_1^p : X \to X$. $X^{\mathbb{N}}$ is endowed with the product $\sigma$-algebra $\tilde{\mathcal{F}}$ and the probability measure $\tilde{m}$ which is the image of $m$ by $\Phi$.

Let $\tilde{\mathcal{F}}_p$ be the $\sigma$-algebra on $X^{\mathbb{N}}$ of events depending only on the coordinates $(x_k)_{k \ge p}$. Clearly, $ \tilde{\mathcal{F}}_p = \sigma^{-p} \tilde{\mathcal{F}}$, where $ \sigma : X^{\mathbb{N}} \to X^{\mathbb{N}}$ is the shift map. Let $K : X^n \to \mathbb{R}$ be a separately Lipschitz function, that we extend as a function $K : X^{\mathbb{N}} \to \mathbb{R}$ belonging to $L^1(\tilde{m})$. Finally, we set $K_p = \mathbb{E}_{\tilde{m}}(K|\tilde{F}_p)$.

\begin{Lem} \label{lem:kp_identity} \em $K_p \circ \Phi = \left( \frac{P_1^p(K \circ \Phi)}{P_1^p \mathds{1}}\right) \circ T_1^p$, $m$-a.e.
\end{Lem}

\begin{proof} 

By Lemma \ref{lem:transfer_prop} (3), we have $K_p = \mathbb{E}_{\tilde{m}}(K| \sigma^{-p} \tilde{\mathcal{F}}) = \left( \frac{P_{\sigma}^p K}{ P_{\sigma}^p \mathds{1}} \right) \circ \sigma^p$. Introduce the map $\Phi_p : X \to X^{\mathbb{N}}$ defined by $\Phi_p(x) = (x, T_{p+1} x, T_{p+1}^{p+2} x, \ldots) = (T_{p+1}^{p+n} x)_{n \ge 0}$. This map satisfies  $\sigma^p \circ \Phi = \Phi_p \circ T_1^p$. It is not non-singular and then its transfer operator is not well defined, but the previous relation implies that $P_{\Phi_p}(P_1^p f)$ is well defined for all $f \in L^1(m)$, is equal to $P_\sigma^p P_\Phi f$ and satisfies all the properties described in Lemma \ref{lem:transfer_prop}.

We then have 
$$K_p \circ \Phi = \left( \frac{P_{\sigma}^p K}{P_{\sigma}^p \mathds{1}} \right) \circ \Phi_p \circ T_1^p = \left( \frac{P_{\sigma}^p P_{\Phi} (K \circ \Phi)}{P_{\sigma}^p P_{\Phi} \mathds{1}} \right) \circ \Phi_p \circ T_1^p,$$
by Lemma \ref{lem:transfer_prop} (2), since $P_{\Phi} \mathds{1} = \mathds{1}$.

Moreover, since $P_{\sigma}^p P_{\Phi} = P_{\Phi_p} P_1^p$, one can use Lemma \ref{lem:transfer_prop} (3) to obtain 
$$K_p \circ \Phi = \left( \frac{P_{\Phi_p}P_1^p (K \circ \Phi)}{P_{\Phi_p} P_1^p \mathds{1}} \right) \circ \Phi_p \circ T_1^p = \left( \frac{( (P_{\Phi_p} \mathds{1}) \circ \Phi_p) \mathbb{E}_m(P_1^p(K \circ \Phi)| (\Phi_p)^{-1} \tilde{ \mathcal{F}}) }{( (P_{\Phi_p} \mathds{1}) \circ \Phi_p)\mathbb{E}_m(P_1^p \mathds{1})| (\Phi_p)^{-1} \tilde{ \mathcal{F}}) } \right) \circ T_1^p.$$

The lemma is proved as soon as we remark that $(\Phi_p)^{-1} \tilde{\mathcal{F}} = \mathcal{F}$.
\end{proof}

Finally, this lemma implies \eqref{eqkp}. Indeed, if we define $\pi_p : X^{\mathbb{N}} \to X$ to be the map that sends $\underline{x}$ to its $p$-th coordinate $x_p$, then $T_1^p = \pi_p \circ \Phi$, and the lemma asserts that $K_p \circ \Phi =  \left( \frac{P_1^p(K \circ \Phi)}{P_1^p \mathds{1}}\right) \circ \pi_p \circ \Phi$, $m$-a.e. 

This implies that $K_p = \left( \frac{P_1^p(K \circ \Phi)}{P_1^p \mathds{1}}\right) \circ \pi_p$, $\tilde{m}$-a.e., and proves the relation, since $$\left( \frac{P_1^p(K \circ \Phi)}{P_1^p \mathds{1}}\right) \circ \pi_p(\underline{x}) = \frac{1}{P_1^p \mathds{1}(x_p)} \sum_{T_1^p y = x_p}  \frac{K(y, T_1 y, \ldots, T_1^{p-1} y, x_p, \ldots )}{| (T_1^p)'(y)|}.$$

To see that, set $A = \{K_p \circ \Phi =  \left( \frac{P_1^p(K \circ \Phi)}{P_1^p \mathds{1}}\right) \circ \pi_p \circ \Phi\}$. We have $A \in \mathcal{F}$, and $m(A) = 1$ by Lemma \ref{lem:kp_identity}. Since $(X, \mathcal{F})$ and $(X^{\mathbb{N}}, \tilde{\mathcal{F}})$ are both Polish spaces endowed with their Borel $\sigma$-algebras, and the map $\Phi : X \to X^{\mathbb{N}}$ is injective and measurable, we obtain that $\Phi(A) \in \tilde{\mathcal{F}}$ by Theorem 8.3.7 in \cite{Cohn}. Then $\tilde{m}(\Phi(A)) = m(\Phi^{-1}(\Phi(A))) \ge m(A) = 1$, which shows that $K_p = \left( \frac{P_1^p(K \circ \Phi)}{P_1^p \mathds{1}}\right) \circ \pi_p$, $\tilde{m}$-a.e., since the set on which this relation holds contains $\Phi(A)$.

\section*{Acknowledgments} The authors would like to thank the anonymous referee for useful comments and in particular for bringing to our knowledge the paper \cite{DM} and for suggesting another proof for the upper bound on the Kantorovich distance. RA is also grateful to S\'ebastien Gou\"{e}zel who suggested him an improvement for Proposition \ref{pro:shadow}. RA would like to thank Carlangelo Liverani for a discussion on the condition {\bf (Min)} and for suggesting him to use his paper \cite{Liv95} and Sandro Vaienti for several discussions and constant encouragement. RA was partially supported by Conseil R\'egional Provence-Alpes-C\^ote d'Azur, the ANR-Project {\em Perturbations}, by the PICS (Projet International de Coop\'eration Scientifique), Propri\'et\'es statistiques des syst\`emes dynamiques deterministes et al\'eatoires, with the University of Houston, n. PICS05968 and by the European Advanced Grant Macroscopic Laws and Dynamical Systems (MALADY) (ERC AdG 246953). Most of this work was done when RA was affiliated to Aix Marseille Universit\'e, CNRS, CPT, UMR 7332, 13288 Marseille, France and Universit\'e de Toulon, CNRS, CPT, UMR 7332, 83957 La Garde, France. RA would also like to acknowledge the Newton Institute of Mathematics at Cambridge, where part of this work was done.  JR was partially supported by CNPq and FAPESB. JR would also like to acknowledge the Universit\'e de Toulon, where part of this work was done.

\end{document}